\documentclass[12pt]{amsart}
\usepackage{etex}
\usepackage{amsfonts, amssymb, latexsym}
\usepackage{amsthm,xparse}

\usepackage{amscd,amssymb}
\usepackage{verbatim}
\usepackage{pstricks}
\usepackage{pst-node}
\usepackage[mathscr]{euscript}
\usepackage{tikz}
\usepackage{tikz-cd}
\usepackage{tkz-euclide}
\usetikzlibrary{matrix,arrows, calc,decorations.pathmorphing,shapes}
\usetikzlibrary{matrix,arrows,backgrounds,shapes.misc,shapes.geometric,patterns,calc,positioning}
\usetikzlibrary{chains,positioning,fit}
\usetikzlibrary{decorations.pathreplacing,calligraphy, decorations.pathmorphing}
\usetikzlibrary{arrows.meta}

\usepackage[normalem]{ulem}
\usepackage{varwidth}
\usepackage{leftidx}

\usepackage{wrapfig,caption}
\usepackage{xfp} 
\usepackage{soul}
\usepackage{float}

\usepackage[colorlinks=true, pdfstartview=FitV, linkcolor=blue, citecolor=blue, urlcolor=blue]{hyperref}

\usepackage{xcolor, fancyhdr}

\definecolor{darkmode}{RGB}{32, 31, 30}
\definecolor{gblue}{RGB}{190,209,210}
\definecolor{oblue}{RGB}{114, 160, 193}
\definecolor{dartmouthgreen}{rgb}{0.05, 0.5, 0.06}

\newsymbol\pp 1275

\usepackage{fullpage}
\usepackage{graphicx}
\usepackage{enumerate}
\usepackage{expl3}

\def\VR{\kern-\arraycolsep\strut\vrule &\kern-\arraycolsep}
\def\vr{\kern-\arraycolsep & \kern-\arraycolsep}

\newtheorem{theorem}{Theorem}[section]

\newtheorem{lemma}[theorem]{Lemma}
\newtheorem{prop}[theorem]{Proposition}
\newtheorem{corollary}[theorem]{Corollary}

\theoremstyle{definition}
\newtheorem{definition}[theorem]{Definition}

\newtheorem{rmk}[theorem]{Remark}

\newtheorem{qu}[theorem]{Question}

\newtheorem*{rmknonum}{Remark}

\newtheorem{obs}[theorem]{Observation}

\newtheorem{ex}[theorem]{Example}

\newcommand{\Hom}{\operatorname{Hom}}
\newcommand{\End}{\operatorname{End}}

\newcommand{\rep}{\operatorname{rep}}

\newcommand{\SI}{\operatorname{SI}}
\newcommand{\SL}{\operatorname{SL}}
\newcommand{\GL}{\operatorname{GL}}

\newcommand{\ZZ}{\mathbb Z}
\newcommand{\CC}{\mathbb C}
\newcommand{\RR}{\mathbb R}
\newcommand{\NN}{\mathbb N}

\newcommand{\be}{\begin{enumerate}} 
\newcommand{\ee}{\end{enumerate}}

\newcommand{\V}{V}

\newcommand{\Id}{\mathbf{Id}}

\newcommand{\ddim}{\operatorname{\mathbf{dim}}}

\newcommand{\Q}{\mathbf{\mathcal{Q}}}

\newcommand{\p}{\mathcal{P}}

\newcommand{\s}{\mathcal{S}}

\newcommand{\num}{\mathbf{T}}

\newcommand{\river}{\mathcal{T}}
\newcommand{\bfa}{\mathbf{a}}
\newcommand{\bfb}{\mathbf{b}}
\newcommand{\one}{\mathbf{1}}

\newcommand{\R}{\operatorname{\mathcal{R}}}

\newcommand\restr[2]{{
  \left.\kern-\nulldelimiterspace 
  #1 
  \vphantom{\big|} 
  \right|_{#2} 
  }}

\tikzset{snake it/.style={decorate, decoration=snake}}

\usetikzlibrary{decorations.pathreplacing,
                calligraphy}
\tikzset{
B/.style = {decorate,
            decoration={calligraphic brace, amplitude=8pt,
            raise=15pt, mirror},
            very thick,
            pen colour=black},
dot/.style = {circle, fill, inner sep=2pt, outer sep=0pt}
        }

\usetikzlibrary{decorations.shapes}

\tikzset{decorate sep/.style 2 args=
{decorate,decoration={shape backgrounds,shape=circle,shape size=#1,shape sep=#2}}}

\usetikzlibrary{quotes}

\begin{document}

\title{Counting 3-way contingency tables via quiver semi-invariants}

\author{Calin Chindris}
\address{University of Missouri-Columbia, Mathematics Department, Columbia, MO, USA}
\email[Calin Chindris]{chindrisc@missouri.edu}

\author{Deepanshu Prasad}
\address{International Center for Mathematical Sciences, Institute of Mathematics and
Informatics, Bulgarian Academy of Sciences, Acad. G. Bonchev Str., Bl. 8, Sofia
1113, Bulgaria}
\email[Deepanshu Prasad]{deepanshu.prasad@gmail.com}

\date{\today}
\subjclass[2010]{16G20, 13A50, 14L24}
\keywords{Contingency tables, Littlewood-Richardson coefficients, quiver exceptional sequences, quiver semi-invariants}

\begin{abstract} Let $\num_{\bfa,\bfb}$ be the number of $3$-way contingency tables of size $m \times n \times p$ with two of its three plane-sum margins fixed by $\bfa=(a_1, \ldots, a_m) \in \NN^m$ and $\bfb=(b_1, \ldots, b_n) \in \NN^n$.  When $p=1$, this is the number of $m \times n$ non-negative integer matrices whose row and column sums are fixed by $\bfa$ and $\bfb$.

In this paper, we study the numbers $\num_{\bfa, \bfb}$ through the lens of quiver invariant theory. Let $\Q^p_{m,n}$ be the $p$-complete bipartite quiver with $m$ source vertices, $n$ sink vertices, and $p$ arrows from each source to each sink. Let $\one$ denote the dimension vector of $\Q^p_{m,n}$ that takes value $1$ at every vertex of $\Q^p_{m,n}$, and let $\theta_{\bfa, \bfb}$ denote the integral weight that assigns $a_i$ to the $i^{th}$ source vertex and $-b_j$ to the $j^{th}$ sink vertex of $\Q^p_{m,n}$. 

We begin by realizing $\num_{\bfa, \bfb}$ as the dimension of the space of semi-invariants associated to $(\Q^p_{m,n}, \one, \theta_{\bfa, \bfb})$. Using this connection and methods from quiver invariant theory, we show that $\num_{\bfa,\bfb}$ is a parabolic Kostka coefficient. In the case $p=1$, this recovers the formula for the number of the $m \times n$ contingency tables with row and column sums fixed by $\bfa$ and $\bfb$, which in the classical $2$-way setting can also be obtained via the Robinson-Schensted-Knuth correspondence.

\end{abstract}

\maketitle
\setcounter{tocdepth}{1}
\tableofcontents

\section{Introduction} 
\label{intro-sec}

Let $m$, $n$, and $p$ be positive integers, and let $\bfa=(a_1, \ldots, a_{m}) $ and $\bfb=(b_1, \ldots, b_{n}) $ be two vectors with non-negative integer coefficients such that 
$$N:=\sum_{i=1}^{m} a_i=\sum_{j=1}^{n} b_j.$$

Let $\num_{\bfa,\bfb}$ be the number of all $3$-way contingency tables $X=(x_{ijk})_{(i,j,k) \in [m] \times [n] \times [p]}$ whose entries $x_{i,j,k}$ are non-negative integers such that
\begin{equation}\label{eqn-margin-1}
\sum_{j =1}^{n} \sum_{k=1}^{p} x_{ijk}=a_i, \forall i \in [m],
\end{equation}
and 
\begin{equation}\label{eqn-margin-2}
\sum_{i =1}^{m} \sum_{k=1}^{p} x_{ijk}=b_j, \forall j \in [n].
\end{equation}

In what follows, $K_{\lambda, \R}$ denotes the parabolic Kostka coefficient associated to a partition $\lambda$ and a sequence $\R=\{ R_1, \ldots, R_s \}$ of rectangular partitions. Thus $K_{\lambda, \R}$ is the multiplicity of the irreducible representation of $\GL(r)$ of highest weight $\lambda$ in the tensor product of the irreducible representations with highest weights $R_1, \ldots, R_s$. We assume that $\lambda$ has at most $r$ non-zero parts and that each $R_i$ has height at most $r$.

\begin{theorem} \label{thm:main-1} With the notation as above,
\begin{equation}\label{eqn:main-formula-parab-Kostka}
\num_{a,b}=K_{\lambda, \R},
\end{equation}
where $\lambda=((p N)^{p m})$ and $\R$ is the following sequence of rectangular partitions
\begin{equation}\label{eqn-1}
\R=\{ \underbrace{(N^{(p-1)m}), \ldots, (N^{(p-1)m})}_{p \text{~times}}, (a_1^{pm-1}), \ldots, (a_m^{pm-1}), (b_1), \ldots, (b_n) \}.
\end{equation}
\end{theorem}

To prove our formula, we first express $\num_{\bfa, \bfb}$ as the dimension of a weight space of semi-invariants for a $p$-complete bipartite quiver. We then apply general reduction techniques, such as the Embedding Theorem \ref{thm:embedding} for quiver semi-invariants, to reduce the problem to computing semi-invariants for a star quiver. These, in turn, can be expressed as parabolic Kostka numbers.

We note that when $p=1$, $\num_{\bfa, \bfb}$ is the number of $m \times n$ non-negative integer matrices whose row and column sums are fixed by $\bfa$ and $\bfb$. In this case, $(\ref{eqn:main-formula-parab-Kostka})$ recovers the classical formula arising from the Robinson–Schensted–Knuth correspondence, but by a quiver-invariant-theoretic argument that does not use the correspondence itself; see Corollary \ref{coro:RSK-corollary} for full details.

\section{Weight spaces of quiver semi-invariants: the tools}
Throughout, we work over the field $\CC$ of complex numbers and denote by $\NN=\{0,1,\dots \}$. For a positive integer $L$, we denote by $[L]=\{1, \ldots, L\}$.

A \emph{quiver} $\Q=(\Q_0,\Q_1,t,h)$ consists of two finite sets $\Q_0$ (vertices) and $\Q_1$ (arrows) together with two maps $t:\Q_1 \to \Q_0$ (tail) and $h:\Q_1 \to \Q_0$ (head). We represent $\Q$ as a directed graph with set of vertices $\Q_0$ and directed edges $a:ta \to ha$ for every $a \in \Q_1$.  

A \emph{representation} of $\Q$ is a family $V=(V_x, V_a)_{x \in \Q_0, a\in \Q_1}$, where $V_x$ is a finite-dimensional $\CC$-vector space for every $x \in \Q_0$, and $V_a: V_{ta} \to V_{ha}$ is a $\CC$-linear map for every $a \in \Q_1$. After fixing bases for the vector spaces $\V_x$, $x \in \Q_0$, we often think of the linear maps $\V_a$, $a \in \Q_1$, as matrices of appropriate size. 

A\emph{ morphism} $\varphi: V \rightarrow W$ between two representations is a collection $(\varphi_x))_{x \in \Q_0}$ of $\CC$-linear maps with $\varphi_x \in \Hom_{\CC}(V_x, W_x)$ for every $x \in \Q_0$, and such that $\varphi_{ha} \circ V_{a}=W_{a} \circ \varphi_{ta}$ for every $a \in \Q_1$. The $\CC$-vector space of all morphisms from $V$ to $W$ is denoted by $\Hom_\Q(V, W)$.

The \emph{dimension vector} $\ddim V \in \NN^{\Q_0}$ of a representation $V$  is defined by $(\ddim V)_x:=\dim_\CC V_x$ for all $x \in \Q_0$. By a \emph{dimension vector} of $Q$, we simply mean an $\NN$-valued function on the set of vertices $\Q_0$. We say a dimension vector $\beta$ is \emph{sincere} if $\beta_x>0$ for every $x \in \Q_0$. The simple dimension vector at $x \in \Q_0$, denoted by $e_x$, is defined by $e_x(y)=\delta_{x,y}$, $\forall y \in \Q_0$, where $\delta_{x,y}$ is the Kronecker symbol. 

The \emph{Euler form} (also known as the Ringel form) of $\Q$ is the bilinear form on $\ZZ^{\Q_0}$ defined by
$$
\langle \alpha, \beta \rangle:=\sum_{x \in \Q_0}\alpha_x\beta_x-\sum_{a \in \Q_1} \alpha_{ta}\beta_{ha}, \; \forall \alpha, \beta \in \ZZ^{\Q_0}.
$$

From now on, we assume that all of our quivers are finite, connected, and acyclic. Then, for any integral weight $\sigma \in \ZZ^{\Q_0}$, there exists a unique $\alpha \in \ZZ^{\Q_0}$ such that $\sigma_x=\langle \alpha, e_x \rangle$, $\forall x \in \Q_0$. 

Let $\beta$ be a sincere dimension vector of a quiver $\Q$, and let us consider the\emph{ representation space} of $\beta$-dimensional representations of $Q$,
$$
\rep(\Q, \beta):=\prod_{a \in \Q_1} \CC^{\beta_{ha} \times \beta_{ta}}.
$$ 
The base change group $\GL(\beta):=\prod_{x \in \Q_0} \GL(\beta_x)$ acts on $\rep(\Q, \beta)$ by simultaneous conjugation, \emph{i.e.}, for $g=(g_x)_{x \in \Q_0} \in \GL(\beta)$ and $W=(W_a)_{a \in \Q_1} \in \rep(\Q,\beta)$, we define $g \cdot W \in \rep(\Q, \beta)$ by
\[
(g \cdot W)_a:=g_{ha}\cdot W_a \cdot g_{ta}^{-1}, \; \forall a \in \Q_1.
\]

\noindent
This action descends to that of the subgroup
\[
\SL(\beta) := \prod_{x \in \Q_0} \SL(\beta_x),
\]
giving rise to a highly non-trivial ring of semi-invariants $\SI(Q,\beta):= \CC[\rep(\Q,\beta)]^{\SL(\beta)}$. (We point out that since $Q$ is assumed to be acyclic, the invariant ring $\CC[\rep(\Q, \beta)]^{\GL(\beta)}$ is precisely $\CC$.) Since $\GL(\beta)$ is linearly reductive and $\SL(\beta)$ is its commutator subgroup, we have the weight space decomposition 
\[
\SI(\Q,\beta) = \bigoplus_{\chi \in X^*(\GL(\beta))} \SI(\Q,\beta)_{\chi},
\]
where $X^*(\GL(\beta))$ is the group of rational characters of $\GL(\beta)$ and 
\[
\SI(\Q,\beta)_{\chi}:= \{f \in \CC[\rep(\Q,\beta)] \mid g \cdot f = \chi(g)f, \, \forall g \in \GL(\beta)\}
\]
is the space of \emph{semi-invariants of weight $\chi$}. 

Every \emph{integral weight} $\sigma \in \ZZ^{\Q_0}$ defines a character $\chi_{\sigma}$ of $\GL(\beta)$  by 
\begin{equation}
\chi_{\sigma}(g):=\prod_{x \in \Q_0} (\det g_x)^{\sigma_x} \qquad \text{for all } g=(g_x)_{x \in \Q_0} \in \GL(\beta).
\end{equation}
Moreover, every character of $\GL(\beta)$ is of the form $\chi_{\sigma}$ for some integral weight $\sigma \in \ZZ^{\Q_0}$. If $\beta$ is sincere,  this gives a one-to-one correspondence,  so we may identify the character group with $\ZZ^{\Q_0}$.  In what follows, we write $\SI(\Q, \beta)_{\sigma}$ for $\SI(\Q, \beta)_{\chi_\sigma}$.

We state a reduction tool that will come on handy when proving Theorem \ref{thm:main-1}. It allows us to remove a vertex of weight zero, provided its dimension is at least that of the head of the unique outgoing arrow.

\begin{lemma}[\textbf{Removing vertices of zero weight}; see Lemma 4.6 in \cite{CC6}]\label{lemma:remove-zero-weight}
Let $\Q$ be a quiver and $v_0$ a vertex such that near $v_0$, $\Q$ looks like:
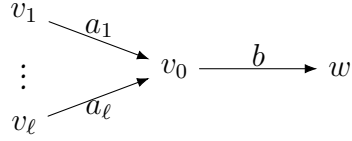
\[
\begin{tikzpicture}[>=Latex, line cap=round, line join=round]
  \node (v0) at (0,0) {$v_0$};
  \node (w)  at (2.2,0) {$w$};

  \node (v1) at (-2.0, 0.75) {$v_1$};
  \node (vl) at (-2.0,-0.75) {$v_\ell$};

  \node at (-2.0,0) {$\vdots$};

  \draw[->] (v1) -- node[above, inner sep=1pt] {$a_1$} (v0);
  \draw[->] (vl) -- node[below, inner sep=1pt] {$a_\ell$} (v0);

  \draw[->] (v0) -- node[above, inner sep=1pt] {$b$} (w);
\end{tikzpicture}
\]

Suppose that $\beta$ is a dimension vector and $\sigma$ is a weight such that
\[
\beta(v_0)\ge \beta(w)
\qquad\text{and}\qquad
\sigma(v_0)=0.
\]
Let $\overline{Q}$ be the quiver defined by $\overline{\Q}_0 = \Q_0\setminus\{v_0\}$ and
\[
\overline{\Q}_1
= \bigl(\Q_1\setminus\{b,a_1,\dots,a_\ell\}\bigr)\,\cup\,\{ba_1,\dots,ba_\ell\}.
\]
If $\overline{\beta}=\beta|_{\overline{Q}}$ and $\overline{\sigma}=\sigma|_{\overline{\Q}}$ are the restrictions
of $\beta$ and $\sigma$ to $\overline{\Q}$ then
\[
\SI(\Q,\beta)_\sigma \;\cong\; \SI(\overline{\Q},\overline{\beta})_{\overline{\sigma}}.
\]
\end{lemma}

\begin{rmk} 
We note that, although the proof in \cite{CC6} is rather short, its main ingredient is the Fundamental Theorem for $\GL(n)$. 
\end{rmk}

\subsection{Network flow polytopes} When the dimension vector is equal to one at every vertex of $\Q$, one can describe the dimensions of the corresponding spaces of semi-invariants in terms of network flows. This description that allows us to express $\num_{\bfa, \bfb}$ as dimensions of weight spaces of quiver semi-invariants.

Let $\sigma \in \ZZ^{\Q_0}$ be an integral weight of $\Q$. The\emph{ network flow polytope} associated to $(\Q, \sigma) $ is defined by
\begin{equation}\label{defn:network-flow-polytope}
\p_\sigma:=\left\{(x_a)_{a \in \Q_1} \in \RR^{\Q_1}_{\geq 0} \;\middle|\;  \sum_{\substack{a\in \Q_1\\ t a = x}} x_a
\;-\;
\sum_{\substack{a\in \Q_1\\ h a = x}} x_a=\sigma_x \qquad \text{for all } x \in \Q_0  \right\}.
\end{equation}

\begin{lemma}[\textbf{Network flows from quiver semi-invariants}; see Lemma in \cite{CC6}]\label{lemma:flows-semi-inv-general} Let $\mathbf{1}$ be the dimension vector that is equal to one at every vertex of $\Q$, and let $\sigma \in \mathbb{Z}^{\Q_0}$ be a weight. Then
\[
\dim_\CC \SI(\Q,\mathbf{1})_\sigma=
\left| \p_{\sigma}\cap \ZZ^{\Q_0} \right|.
\]
\end{lemma}

\subsection{Reflection transformations} We will also make use of reflection transformations to establish $(\ref{eqn:main-formula-parab-Kostka})$. Let $\beta \in \ZZ^{\Q_0}_{\geq 0}$ be a dimension vector, $\sigma \in \ZZ^{\Q_0}$, an integral weight and a $x \in \Q_0$ a vertex. We define $s_x \Q$ to be the quiver obtained from $\Q$ by reversing all arrows that start or end in $x$. We also define $s_x\beta \in \ZZ^{\Q_0}$ by
\[
(s_x \beta)_y=
\begin{cases}
\beta_y, & \text{if } x \ne y,\\[4pt]
-\beta_x + \displaystyle\sum_{\text{edges } x\text{---}z} \beta_z, & \text{if } x = y,
\end{cases}
\]
and $s_x \sigma \in \ZZ^{\Q_0}$ by
\[
(s_x \sigma)_y=
\begin{cases}
-\sigma_x, & \text{if } y=x,\\[4pt]
\sigma_y + \sigma_x b_{xy}, & \text{if } x \neq y,
\end{cases}
\]
where $b_{xy}$ is the number of edges between $x$ and $y$.

\begin{theorem} Suppose that $x$ is a sink or a source vertex and $s_x\beta$ is a dimension vector, \emph{i.e.}, $s_x\beta \in \ZZ_{\geq 0}^{\Q_0}$. Then
\begin{equation}\label{eqn:semi-inv-reflections}
\SI(\Q, \beta)_{\sigma} \simeq \SI(s_x \Q, s_x\beta)_{s_x\sigma}.
\end{equation}
\end{theorem}

\subsection{Quiver exceptional sequences} The notion of a quiver exceptional sequence, which we review below, plays a key role in our computations, since it allows us to reduce the problem to computing the dimensions of weight spaces of semi-invariants for a star quiver. These can then be expressed as parabolic Kostka coefficients.

In what follows, by a \emph{Schur representation} $\V$ of $\Q$, we mean a representation such that
$\dim \End_\Q(\V)=1$, that is, $\End_\Q(\V)=\{(\lambda \Id_{\V(x)})_{x\in \Q_0}\mid \lambda\in \CC\}$.

For two dimension vectors $\alpha$ and $\beta$, we define $(\alpha \circ \beta)_\Q:=\dim \SI(\Q, \beta)_{\langle \alpha, \cdot \rangle}$. (We drop the subscript $\Q$ whenever $\Q$ is understood from the context.) It follows from the main results in \cite{DW1} that $\alpha\circ\beta \neq 0$ if and only if $\langle \alpha,\beta\rangle =0$ and $\Hom_\Q(V,W)=0$ for some representations $V$ and $W$ of dimension vectors $\alpha$ and $\beta$, respectively.

\begin{definition}[\textbf{Quiver Exceptional Sequences}]\label{def:quiver-exceptional-sequence} A sequence
\[
\mathcal{E}=(\varepsilon_1,\ldots,\varepsilon_N)
\]
of dimension vectors of $\Q$ is said to be a \emph{quiver exceptional sequence} if:
\begin{enumerate}
\item $\langle \varepsilon_i,\varepsilon_i\rangle =1$ and $\varepsilon_i$ is the dimension vector of a Schur representation for all \(i\in [N]\);
    
\item $\langle \varepsilon_i,\varepsilon_j\rangle \le 0$ and $\varepsilon_j\circ \varepsilon_i \neq 0$ for all $1\le i<j\le N$.
\end{enumerate}
\end{definition}

To any quiver exceptional sequence $\mathcal{E}=(\varepsilon_1,\ldots,\varepsilon_N)$, we associate the quiver $\Q(\mathcal{E})$ with vertices $\{1,\ldots,N \}$ and
\[
-\langle \varepsilon_i,\varepsilon_j\rangle
\]
arrows from vertex \(i\) to vertex \(j\) for all \(1\le i\ne j\le N\). Let
\begin{equation}\label{eqn:transf-qes}
I:\RR^N\to \RR^{\Q_0}
\end{equation}
be the map defined by
\[
I(\gamma_1,\ldots,\gamma_N):=\sum_{i=1}^N \gamma_i\varepsilon_i
\]
for all $\gamma=(\gamma_1,\ldots,\gamma_N)\in \RR^N$.

We are now ready to state Derksen--Weyman's Embedding Theorem.

\begin{theorem}[\textbf{The Embedding Theorem for Quiver Semi-Invariants}; see \cite{DW2}]\label{thm:embedding}
Let
\[
\mathcal{E}=(\varepsilon_1,\ldots,\varepsilon_N)
\]
be a quiver exceptional sequence for $\Q$. If \(\alpha\) and \(\beta\) are two dimension vectors of $\Q(\mathcal{E})$, then
\[
(\alpha\circ\beta)_{\Q(\mathcal{E})}=(I(\alpha)\circ I(\beta))_\Q.
\]
\end{theorem}

\section{The number of $3$-way contingency tables as parabolic Kostka coefficients} In this section, we will work with the $p$-complete bipartite quiver with set of source vertices $\{x_1 \ldots, x_m\}$, set of sink vertices $\{y_1, \ldots y_n \}$, and such that there are $p$ arrows from any source vertex $x_i$ to any sink vertex $y_j$:
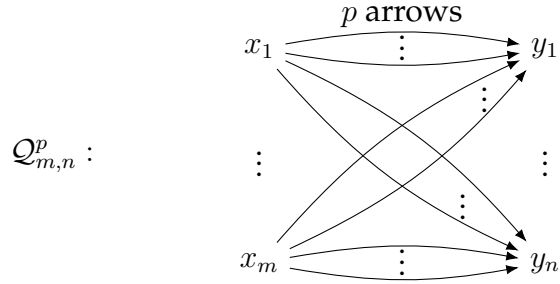
\begin{figure}[H]
\centering
\begin{tikzpicture}[>=Latex, scale=0.95]
    \node (x1) at (0,2) {$x_1$};
    \node (xv) at (0,0.5) {$\vdots$};
    \node (xm) at (0,-1) {$x_m$};

    \node (y1) at (4,2) {$y_1$};
    \node (yv) at (4,0.5) {$\vdots$};
    \node (yn) at (4,-1) {$y_n$};

    \node at (-2.9,0.5) {$\Q^p_{m,n}:$};

    \draw[->, bend left=10]  (x1) to (y1);
    \draw[->, bend right=10] (x1) to (y1);
    \node at (2,2.12) {$\vdots$};
    \node at (2,2.45) {$p$ arrows};

    \draw[->, bend left=10]  (xm) to (yn);
    \draw[->, bend right=10] (xm) to (yn);
    \node at (2,-0.88) {$\vdots$};

    \draw[->, bend left=12]  (x1) to (yn);
    \draw[->, bend right=12] (x1) to (yn);
    \node at (2.85,-0.10) {$\vdots$};

    \draw[->, bend left=12]  (xm) to (y1);
    \draw[->, bend right=12] (xm) to (y1);
    \node at (3.15,1.40) {$\vdots$};
\end{tikzpicture}
\end{figure}

Recall that $\theta_{\bfa, \bfb}$ is the integral weight that assigns $a_i$ to the $i^{th}$ source vertex and $-b_j$ to the $j^{th}$ sink vertex of $\Q^p_{m,n}$.  Then the network flow polytope $\p_{\theta_{\bfa, \bfb}}$ is the set of all $3$-way contingency tables with margins given by $(\ref{eqn-margin-1})$ and $(\ref{eqn-margin-2})$.  It now follows from Lemma \ref{lemma:flows-semi-inv-general} that
\begin{equation}\label{eqn:formual-numCT-semi-inv-thin}
\num_{\bfa,\bfb}=\dim \SI(\Q^p_{m,n}, \mathbf{1})_{\theta_{\bfa,\bfb}}.
\end{equation}

Next, we explain how to simplify the task of computing semi-invariants for $\Q^p_{m,n}$ via the Embedding Theorem \ref{thm:embedding}. To this end, we introduce the following star quiver
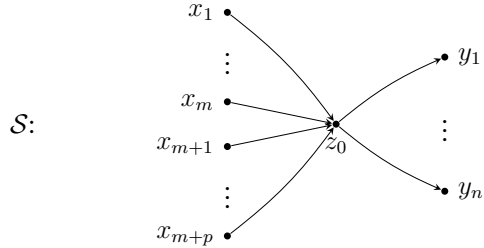
\begin{center}
\scalebox{0.85}{%
\begin{tikzpicture}[>=stealth]
  \tikzset{vertex/.style={circle,fill=black,inner sep=1.1pt}}

  \node at (-4.9,-0.35) {$\s$:};

  \node[vertex,label=left:$x_1$]     (x1)   at (-1.7,  1.4) {};
  \node                              at (-1.7,  0.7) {$\vdots$};
  \node[vertex,label=left:$x_m$]     (xm)   at (-1.7,  0.0) {};
  \node[vertex,label=left:$x_{m+1}$] (xmp1) at (-1.7, -0.7) {};
  \node                              at (-1.7, -1.4) {$\vdots$};
  \node[vertex,label=left:$x_{m+p}$] (xmp)  at (-1.7, -2.1) {};

  \node[vertex,label=below:$z_0$] (z0) at (0,-0.35) {};

  \node[vertex,label=right:$y_1$] (y1) at (1.7,  0.7) {};
  \node                           at (1.7, -0.35) {$\vdots$};
  \node[vertex,label=right:$y_n$] (yn) at (1.7, -1.4) {};

  \draw[->,bend left=8]  (x1)   to (z0);
  \draw[->]              (xm)   -- (z0);
  \draw[->]              (xmp1) -- (z0);
  \draw[->,bend right=8] (xmp)  to (z0);

  \draw[->,bend left=8]  (z0) to (y1);
  \draw[->,bend right=8] (z0) to (yn);
\end{tikzpicture}%
}
\end{center}
Let $\beta$ be the dimension vector of $\s$ defined by
\[
\beta(x_\ell)=
\begin{cases}
1 & \text{if } \ell \in [m] \\
m & \text{if } \ell \in \{m+1, \ldots, m+p\},
\end{cases}
\]
$\beta(z_0)=mp$, and $\beta(y_j)=1, \forall j \in [n]$. We also need the weight $\sigma_{\bfa, \bfb}$ of $\s$ defined by
\[
\sigma_{\bfa,\bfb}(x_\ell)=
\begin{cases}
a_\ell & \text{if } \ell \in [m] \\
N & \text{if } \ell \in \{m+1, \ldots, m+p\},
\end{cases}
\]
and $\sigma_{\bfa, \bfb}(z_0)=-N$, and $\sigma_{\bfa, \bfb}(y_j)=-b_j, \forall j \in [n]$.

\begin{prop}\label{prop:main} Keep the same notation as above. Then
\begin{equation}
\num_{a,b}=\dim_\CC \SI(\s, \beta)_{\sigma_{\bfa, \bfa}}.
\end{equation}
\end{prop}

\begin{proof}
Following \cite{ChiColKli-2025}, we consider the quiver
\begin{center}
\scalebox{0.95}{%
\begin{tikzpicture}[>=stealth]
  \tikzset{vertex/.style={circle,fill=black,inner sep=1.1pt}}

  \node at (-4.1,-0.35) {$\river:$};

  \node[vertex,label=left:$x_1$]     (x1)   at (-1.9,  1.6) {};
  \node                              at (-1.9,  0.9) {$\vdots$};
  \node[vertex,label=left:$x_m$]     (xm)   at (-1.9,  0.2) {};
  \node[vertex,label=left:$x_{m+1}$] (xmp1) at (-1.9, -0.5) {};
  \node                              at (-1.9, -1.2) {$\vdots$};
  \node[vertex,label=left:$x_{m+p}$] (xmp)  at (-1.9, -1.9) {};

  \node[vertex,label=above:$x_0$] (x0) at (0,-0.15) {};
  \node[vertex,label=above:$y_0$] (y0) at (2.2,-0.15) {};

  \node[vertex,label=right:$y_1$] (y1) at (4.4,  0.9) {};
  \node                           at (4.4, -0.15) {$\vdots$};
  \node[vertex,label=right:$y_n$] (yn) at (4.4, -1.2) {};

  \draw[->,bend left=12]  (x1)   to (x0);
  \draw[->]               (xm)   -- (x0);
  \draw[->]               (xmp1) -- (x0);
  \draw[->,bend right=12] (xmp)  to (x0);

  \draw[->] (x0) -- (y0);

  \draw[->,bend left=12]  (y0) to (y1);
  \draw[->,bend right=12] (y0) to (yn);
\end{tikzpicture}%
}
\end{center}

\noindent
Furthermore, for each $i \in [m]$, let $\delta_i$ be the dimension vector defined by
\[
\delta_i(x_0)=p+1,\qquad \delta_i(y_0)=p,\qquad \delta_i(x_i)=1,
\]
\[
\delta_i(x_{m+1})=\cdots=\delta_i(x_{m+p})=1,
\]
and
\[
\delta_i(v)=0\quad\text{for all other vertices }v\in \river_0.
\]

\noindent
It follows from \cite[Proposition 3.4]{ChiColKli-2025} that
\[
\mathcal{E}:=(\delta_1,\ldots,\delta_m,e_{y_1},\ldots,e_{y_n})
\]
is an exceptional sequence of $\river$ with
\[
\river(\mathcal{E})=\Q^p_{m,n}.
\]

\noindent
Now let $I: \RR^{\Q_0} \to \RR^{\river_0}$ be the transformation $(\ref{eqn:transf-qes})$ corresponding to $\river$ and $\mathcal{E}$. Then the dimension vector $\widehat{\beta}:=I(\mathbf{1})$ of $\river$ is given by
\[
\widehat{\beta}(x_i)=1 \quad (i\in [m]),\qquad \widehat{\beta}(x_{m+\ell})=m \quad (\ell \in [p]), \qquad 
\widehat{\beta}(x_0)=(p+1)m,
\]
\[
\widehat{\beta}(y_0)=pm,\qquad
\widehat{\beta}(y_j)=1 \quad (j\in [n]).
\]

\noindent
Next, let us write
\[
\theta_{a,b}=\langle \alpha,-\rangle_{\Q^p_{m,n}},
\]
where $\alpha$ is the dimension vector of $\Q^p_{m,n}$ given by
\[
\alpha(x_i)=a_i \quad (i\in[m]),\qquad
\alpha(y_j)=pN-b_j \quad (j\in[n]).
\]

\noindent
Then the dimension vector $I(\alpha)$ of $\river$ is given by
\[
I(\alpha)(x_i)=a_i \quad (i\in [m]),\qquad
I(\alpha)(x_{m+\ell})=N \quad (\ell \in [p]),
\]
\[
I(\alpha)(x_0)=(p+1)N,\qquad
I(\alpha)(y_0)=pN,
\]
\[
I(\alpha)(y_j)=pN-b_j \quad (j\in [n]).
\]

\noindent
Hence the weight $\widehat{\sigma}_{\bfa,\bfb}:=\langle I(\alpha),-\rangle_{\river}$ of $\river$ is given by
\[
\widehat{\sigma}_{\bfa,\bfb}(x_i)=a_i \quad (i\in [m]),\qquad
\widehat{\sigma}_{\bfa,\bfb}(x_{m+\ell})=N \quad (\ell \in [p]),
\]
\[
\widehat{\sigma}_{\bfa,\bfb}(x_0)=0,\qquad
\widehat{\sigma}_{\bfa,\bfb}(y_0)=-N,
\]
\[
\widehat{\sigma}_{\bfa,\bfb}(y_j)=-b_j \quad (j\in [n]).
\]

\noindent
Applying Theorem \ref{thm:embedding}, we obtain that
\[
\dim \SI(\Q^p_{m,n},\mathbf{1})_{\theta_{\bfa,\bfb}}
=
\dim \SI(\river,\widehat{\beta})_{\widehat{\sigma}_{\bfa,\bfb}}.
\]

\noindent
Since $\widehat{\sigma}_{a,b}(x_0)=0$ and $\widehat{\beta}(x_0)>\widehat{\beta}(y_0)$, it follows from Theorem \ref{lemma:remove-zero-weight} that
\[
\dim \SI(\river,\widehat{\beta})_{\widehat{\sigma}_{\bfa,\bfb}}
=
\dim \SI(\s,\beta)_{\sigma_{\bfa, \bfb}}.
\]
The proof now follows.
\end{proof}

We are now ready to prove Theorem \ref{thm:main-1}. For background material relevant to the computations in the proof below, we refer the reader to \cite[Section 4]{ChiColKli-2025}.

\begin{proof}[Proof of Theorem \ref{thm:main-1}]
We know from Proposition \ref{prop:main} that
\begin{equation}\label{eqn:proof-thm1}
\num_{\bfa,\bfb}=\dim \SI(\s,\beta)_{\sigma_{\bfa,\bfb}}.
\end{equation}
Now, applying $(\ref{eqn:semi-inv-reflections})$ at the source vertices $x_1,\dots,x_{m+p}$, we get that
\begin{equation}\label{eqn:proof-thm2}
\dim \SI(\s,\beta)_{\sigma_{\bfa,\bfb}}=\dim \SI(\widetilde{\s},\widetilde{\beta})_{\widetilde{\sigma}_{\bfa,\bfb}},
\end{equation}
where

\begin{center}
\scalebox{0.85}{%
\begin{tikzpicture}[>=stealth]
  \tikzset{vertex/.style={circle,fill=black,inner sep=1.1pt}}

  \node at (-4.1,-0.35) {$\widetilde{\s}:$};

  \node[vertex,label=left:$x_1$]     (x1)   at (-1.7,  1.4) {};
  \node                              at (-1.7,  0.7) {$\vdots$};
  \node[vertex,label=left:$x_m$]     (xm)   at (-1.7,  0.0) {};
  \node[vertex,label=left:$x_{m+1}$] (xmp1) at (-1.7, -0.7) {};
  \node                              at (-1.7, -1.4) {$\vdots$};
  \node[vertex,label=left:$x_{m+p}$] (xmp)  at (-1.7, -2.1) {};

  \node[vertex,label=below:$z_0$] (z0) at (0,-0.35) {};

  \node[vertex,label=right:$y_1$] (y1) at (1.7,  0.7) {};
  \node                           at (1.7, -0.35) {$\vdots$};
  \node[vertex,label=right:$y_n$] (yn) at (1.7, -1.4) {};

  \draw[->,bend left=8]  (z0) to (x1);
  \draw[->]              (z0) -- (xm);
  \draw[->]              (z0) -- (xmp1);
  \draw[->,bend right=8] (z0) to (xmp);

  \draw[->,bend left=8]  (z0) to (y1);
  \draw[->,bend right=8] (z0) to (yn);
\end{tikzpicture}%
}
\end{center}
\noindent
and
\[
\widetilde{\beta}(x_1)=\cdots=\widetilde{\beta}(x_m)=pm-1,\qquad
\widetilde{\beta}(x_{m+1})=\cdots=\widetilde{\beta}(x_{m+p})=m(p-1),
\]
\[
\widetilde{\beta}(z_0)=pm,\qquad
\widetilde{\beta}(y_1)=\cdots=\widetilde{\beta}(y_n)=1,
\]
and
\[
\widetilde{\sigma}_{a,b}(x_i)=-a_i \qquad (i\in[m]),\qquad
\widetilde{\sigma}_{a,b}(x_{m+\ell})=-N \qquad (\ell\in[p]),
\]
\[
\widetilde{\sigma}_{a,b}(z_0)=pN,\qquad
\widetilde{\sigma}_{a,b}(y_j)=-b_j \qquad (j\in[n]).
\]

To find a closed formula for $\dim \SI(\widetilde{\s},\widetilde{\beta})_{\widetilde{\sigma}_{\bfa,\bfb}}$, we proceed as follows. First, we use Cauchy's formula to decompose $\CC[\rep(\widetilde{\s},\widetilde{\beta})]$ into a direct sum of irreducible representations of $\GL(\widetilde{\beta})$. Then, we consider the ring of semi-invariants
\[
\SI(\widetilde{\s},\widetilde{\beta}):=\CC[\rep(\widetilde{\s},\widetilde{\beta})]^{\SL(\widetilde{\beta})}
\]
and sort out those semi-invariants of weight $\widetilde{\sigma}_{\bfa,\bfb}$. For convenience, we write
\[
V_i=\CC^{\widetilde{\beta}(x_i)},\ i\in[m+p],\qquad V=\CC^{pm},\qquad
W_j=\CC^{\widetilde{\beta}(y_j)}=\CC,\ j\in[n].
\]
Then
\[
\CC[\rep(\widetilde{\s},\widetilde{\beta})]
=
\CC\Biggl[\prod_{i=1}^{m+p}\Hom_{\CC}(V,V_i)\times \prod_{j=1}^{n}\Hom_{\CC}(V,W_j)\Biggr]
\]
\[
=\bigotimes_{i=1}^{m+p} S(V\otimes V_i^*)\otimes \bigotimes_{j=1}^{n} S(V\otimes W_j^*)
\]
\[
=\bigoplus \bigotimes_{i=1}^{m+p} S^{\lambda(i)}V_i^*\otimes \bigotimes_{j=1}^{n} S^{\mu(j)}W_j^*
\otimes
\left(S^{\lambda(1)}V\otimes \cdots \otimes S^{\lambda(m+p)}V\otimes S^{\mu(1)}V\otimes \cdots \otimes S^{\mu(n)}V\right),
\]
where the direct sum is over all partitions $\lambda(i)$ and $\mu(j)$ with $\ell(\lambda(i))\leq \dim V_i$, $\forall i\in[m+p]$, and $\ell(\mu(j))\leq 1$, $\forall j\in[n]$.

Thus, $\SI(\widetilde{\s},\widetilde{\beta})$ can be written as
\[
\bigoplus \bigotimes_{i=1}^{m+p} \left(S^{\lambda(i)}V_i^*\right)^{\SL(V_i)}
\otimes
\bigotimes_{j=1}^{n} \left(S^{\mu(j)}W_j^*\right)^{\SL(W_j)}
\otimes
\left(\bigotimes_{i=1}^{m+p} S^{\lambda(i)}V\otimes \bigotimes_{j=1}^{n} S^{\mu(j)}V\right)^{\SL(V)}.
\]
Sorting out those semi-invariants of weight $\widetilde{\sigma}_{\bfa,\bfb}$ completely determines the partitions $\lambda(i)$ and $\mu(j)$. This way we get that $\SI(\widetilde{\s},\widetilde{\beta})_{\widetilde{\sigma}_{\bfa,\bfb}}$ is isomorphic to the weight space of weight $\widetilde{\sigma}_{\bfa,\bfb}(z_0)=pN$ that occurs in the weight space decomposition of
\[
\left(
\underbrace{S^{(N^{p(m-1)})}V\otimes \cdots \otimes S^{(N^{p(m-1))}V}}_{p\text{ times}}
\otimes \bigotimes_{i=1}^{m} S^{(a_i^{pm-1})}V
\otimes \bigotimes_{j=1}^{n} S^{b_j}V
\right)^{\SL(V)}.
\]
Thus, $\SI(\widetilde{\mathcal{S}},\widetilde{\beta})_{\widetilde{\sigma}_{a,b}}$ is isomorphic to
\[
\left(
\underbrace{S^{(N^{p(m-1)})}V\otimes \cdots \otimes S^{(N^{p(m-1)})}V}_{p\text{ times}}
\otimes \bigotimes_{i=1}^{m} S^{(a_i^{pm-1})}V
\otimes \bigotimes_{j=1}^{n} S^{b_j}V
\otimes S^{((pN)^{pm})}V^*
\right)^{\GL(V)}.
\]
The dimension of this weight space is precisely $K_{\lambda,R}$. This together with $(\ref{eqn:proof-thm1})$ and $(\ref{eqn:proof-thm2})$ yields the desired formula for $\num_{\bfa,\bfb}$.
\end{proof}

\begin{corollary}\label{coro:RSK-corollary} The following formula holds.
\begin{equation}\label{eqn:formula-RSK}
\num_{\bfa,\bfb}=\sum_{\mu} K_{\mu, \{(a_1), \ldots, (a_m), \underbrace{(N^m), \ldots, (N^m)}_{p \text{~times}} \}} \cdot K_{\mu, \{(b_1), \ldots, (b_n), (N^{mp})\}},
\end{equation} 
where the sum is over all partitions $\mu$ with at most $mp$ parts.  When $p=1$, this recovers the formula for the number of the $m \times n$ contingency tables with row and column sums fixed by $\bfa$ and $\bfb$, as derived from the Robinson-Schenstead-Knuth correspondence.
\end{corollary}

\begin{proof} We know from Proposition \ref{prop:main} that
\[
\num_{\bfa,\bfb}=\dim \SI(\s,\beta)_{\sigma_{\bfa,\bfb}}.
\]
Using the same strategy as in the proof of Theorem~1.1, we get that $\SI(\s,\beta)_{\sigma_{\bfa,\bfb}}$ is isomorphic to
\begin{equation*}
\left(
S^{a_1}V^*\otimes \cdots \otimes S^{a_m}V^*
\otimes \underbrace{S^{(N^m)}V^*\otimes \cdots \otimes S^{(N^m)}V^*}_{p\text{ times}}
\otimes S^{b_1}V\otimes \cdots \otimes S^{b_n}V\otimes S^{(N^{mp})}V
\right)^{\GL(V)},
\end{equation*}
where $V=\CC^{mp}$. The dimension of this space is precisely the right hand side of $(\ref{eqn:formula-RSK})$.

\end{proof}

\subsection*{Acknowledgements}
D.P. was supported by the Bulgarian National Science Fund Contract No: KP-06-N92/5, the Ministry of Education and Science of the Republic of Bulgaria, grant DO1-239/10.12.2024 and the Simons Foundation grant SFI-MPS-T-Institutes-00007697.

\bibliography{biblio-cont-tables-QSI}\label{biblio-sec}
\bibliographystyle{alpha}

\end{document}